\documentclass[12pt]{article}

\usepackage[top=1.5in, bottom=1.5in, left=1.2in, right=1.2in]{geometry}
\usepackage{marginnote}
\usepackage[active]{srcltx}
\usepackage{amsmath}
\usepackage{amsthm}
\usepackage{amsxtra}
\usepackage{bbm}
\usepackage{graphicx}
\usepackage{amsfonts,amssymb}
\usepackage{latexsym}
\usepackage{float}
\newtheorem{Th}{Theorem}
\newtheorem{Prop}{Proposition}
\newtheorem{Lm}{Lemma}
\newtheorem{Co}{Corollary}

\theoremstyle{definition}
\newtheorem{Def}{Definition}

\newtheorem{Rem}{Remark}

\newcommand{\aut}{\mathrm{Aut}}

\newcommand{\supp}{\mathrm{supp}}

\newcommand{\st}{\mathrm{St}}

\newcommand{\sym}{\mathrm{Sym}}
\newcommand{\rist}{\mathrm{rist}}

\newcommand{\dd}{\mathrm{d}}
\newcommand{\id}{\mathrm{I}}

\newcommand{\Sp}{\mathrm{Span}}

\newcommand{\eg}{\text{e.g.\,}}
\newcommand{\ie}{\text{i.e.\;\,}}
\newcommand{\etc}{\text{etc}}

\makeatletter
\def\blfootnote{\xdef\@thefnmark{}\@footnotetext}
\makeatother

\begin{document}
\title{On spectra of Koopman, groupoid and quasi-regular representations.}
\author{ {\bf Artem Dudko}  \\
                    Stony Brook University, Stony Brook, NY, USA  \\
          artem.dudko@stonybrook.edu \\
         {\bf Rostislav Grigorchuk
         } \\        Texas A\&M University, College Station, TX, USA  \\      grigorch@math.tamu.edu }

\blfootnote{Both authors were supported by the Swiss National Science Foundation}
\blfootnote{The second author was supported by NSF grant DMS-1207699 and NSA grant H98230-15-1-0328}

\date{}

\maketitle

\section{Introduction.}
The study of spectra of operators of unitary group representations has a long history, remarkable achievements and numerous applications. For instance, the famous Kadison-Kaplanski Conjecture which was proven for the case of amenable groups by Higson and Kasparov in \cite{HigsKasp97} asserts that for a torsion free group $G$ and an element $m\in\mathbb C[G]$ of the group algebra of $G$ the spectrum of $\lambda_G(m)$ is connected, where $\lambda_G$ is the left regular representation of $G$. The remarkable Kesten's criterion of amenability and the fundamental property $(T)$ of Kazhdan can be formulated in terms of spectral properties of operators of the form $\lambda_G(m)$. The topic in discussion is related to the spectral theory of graphs and networks, random walks, theory of operator algebras, discrete potential theory, abstract harmonic analysis \etc.

There are three important types of unitary representations associated to a measure class preserving action of a countable group $G$ on a probability space $(X,\mu)$: quasi-regular, Koopman and groupoid representations. The goal of this article is to show that there is a close relation between spectral properties of these three types of representations.


For a subgroup $H<G$ the quasi-regular representation $\rho_{G/H}$ acting on $l^2(G/H)$ is a natural generalization of the regular representation $\lambda_G$. In the case of a group action $(G,X,\mu)$ such representations appear as permutational representations $\rho_x$ in $l^2(Gx)$ for the action of $G$ on orbits $Gx$, $x\in X$. Spectra of quasi-regular representations play an important role in random walks on groups and Schreier graphs (see \eg \cite{Kes59}). 
Quasi-regular representations naturally give rise to Hecke algebras and their representations. 

 The Koopman representation (which we denote by $\kappa$) acts in $L^2(X,\mu)$.
Some important properties of the dynamical system $(G,X,\mu)$, such as ergodicity and weak mixing,  can be reformulated in terms of spectral properties of $\kappa$  (see \eg \cite{BeHaVa}).


If, in addition, $G$ is countable then the groupoid representation $\pi$ is defined in $L^2(\mathcal R,\nu)$, where $\mathcal R\subset X\times X$ is the orbit equivalence relation and $\nu$ is a measure on $\mathcal R$ which is the product of $\mu$ and the counting measure on leaves. Groupoid representations play important role in operator algebras (see \eg \cite{Tak3}) and theory of factor representations and character theory (see \eg \cite{VK} and \cite{DM13 AF}).

 Given a unitary representation $U$ of a group $G$ and an element $m\in\mathbb C[G]$ (or, more generally $m\in l^1(G)$)  define Hecke type operator \begin{equation}\label{EqHecke}U(m)=\sum\limits_{s\in G}m(s)U(s).\end{equation} For an operator $A$ denote by $\sigma(A)$ its spectrum. The main result of the present paper is the following:
\begin{Th}\label{ThMain}
$1)$ For an ergodic measure class preserving action of a countable group $G$ on a standard probability space $(X,\mu)$ and any $m\in\mathbb C[G]$ one has \begin{equation}\label{EqMainCont}\sigma(\kappa(m))\supset\sigma(\rho_x(m))=\sigma(\pi(m))
  \end{equation} for $\mu$-almost all $x\in X$, where $\kappa$ is the Koopman representation, $\pi$ is the groupoid representation associated to the action of $G$ on $X$, $\rho_x$ is the quasi-regular representation associated with the orbit $Gx$.\\
$2)$ If, moreover, $\mu$ is $G$-invariant and non-atomic, then $\sigma(\kappa_0(m))\supset\sigma(\pi(m))$, where $\kappa_0$ is the restriction of $\kappa$ onto the orthogonal complement of constant functions in $L^2(X,\mu)$. \\
$3)$ If, in addition to the conditions of \hskip0.1cm $1)$, $(G,X,\mu)$ is hyperfinite, then \begin{equation}\label{EqMainSim}\sigma(\kappa(m))=\sigma(\pi(m)).\end{equation}
\end{Th}

This result has an interpretation in terms of weak containment of representations. Given a unitary representation $\rho$ let $C_\rho$ denote the $C^*$-algebra generated by operators $\rho(g),g\in G$.  Recall that a unitary representation $\rho$ of a group $G$ is weakly contained in a unitary representation$\eta$ (denoted by $\rho\prec\eta$) if there exists a surjective homomorphism $\phi:C_\eta\to C_\rho$ of $C^*$-algebras such that $\phi(\eta(g))=\rho(g)$ for all $g\in G$. We write $\rho\sim\eta$ if $\rho$ is weakly equivalent to $\eta$ (\ie $\rho\prec\eta$ and $\eta\prec\rho$). An action $(G,X,\mu)$ of a countable group $G$ is called hyperfinite if the orbit equivalence relation associated to this action is hyperfinite with respect to $\mu$ (see \eg \cite{FM1}). Theorem \ref{ThMain} can be formulated in terms of weak containment. Namely, \eqref{EqMainCont} means that
$$\kappa\succ\rho_x\sim\pi$$ for $\mu$-almost all $x\in X$ and \eqref{EqMainSim} means that  $\kappa\sim\pi$.

As an application of relations between spectra of representations we describe the spectra of the torsion group $\mathcal{G}=<a,b,c,d>$ of intermediate growth constructed by the second author in \cite{Gr80} and studied in \cite{Gr84} and other papers. Recall that $\mathcal{G}$ acts naturally on the boundary $\partial T$ of a binary rooted tree $T$ (see \eg \cite{Grig11}). We prove the following:
\begin{Th}\label{ThGammaSpec}
The spectrum of the Cayley graph of $\mathcal{G}$ $($\ie the spectrum of $\lambda_\mathcal{G}(a+b+c+d))$ is $[-2,0]\cup [2,4]$ and coincides with the spectrum of the Schreier graph $\Gamma_x$ of the action of $\mathcal G$ on $\partial T$ for any $x\in\partial T$ (\ie with the spectrum of $\rho_x(a+b+c+d)$).\end{Th}
\noindent In fact, our proof shows that the spectra of $\lambda_\mathcal{G}(t a+b+c+d)$ and $\rho_x(t a+b+c+d)$ coincide for any $t\in\mathbb R$ and almost every $x\in\partial T$ and are equal to a union of two intervals, an interval, or two points. In \cite{GLN} the authors studied the operator $\rho_x(ta+ub+vc+wd)$ for parameters $t,u,v,w\in\mathbb R$ such that $t\neq 0$, $u\neq -v$, $u\neq -w$, $v\neq -w$ and at least two of $v,w,t$ are distinct. They showed that the spectrum of $\rho_x(ta+ub+vc+wd)$ is a Cantor set of Lebesgue measure zero by reduction to the results known for random Schr\"odinger operators and substitutional dynamical systems. 
The corresponding substitution
 $$\tau:a\to aca,b\to d,c\to b, d\to c$$ appears in the presentation
 $$\mathcal G=<a,b,c,d|a^2,b^2,c^2,d^2,bcd,\tau^i((ad)^4),\tau^i((adacac)^4)>$$ found by Lysenok in \cite{Lys}. An interesting question is whether the spectra of $\lambda_\mathcal{G}(ta+ub+vc+wd)$ and $\rho_x(ta+ub+vc+wd)$ coincide for arbitrary parameters $t,u,v,w\in\mathbb R$.

 Notice that there are not many examples of groups for which the spectrum of the Cayley graph has been calculated. Theorem \ref{ThMain} is the first case when the spectrum of the Cayley graph is computed for a group of intermediate growth. The coincidence of the spectra of Schreier graphs $\Gamma_x$ of the action of $\mathcal G$ on $\partial T$ and the Cayley graph of $\mathcal G$ is very surprising since the $\Gamma_x$ are of linear growth and are very different from the Cayley graph of $\mathcal G$.

 Observe that $\mathcal G$ has an abelian extension $\tilde{\mathcal G}$ which is a torsion free group of intermediate growth generated by four elements $\tilde a,\tilde b,\tilde c,\tilde d$ (see \cite{Gr84}). From the amenability of $\tilde{\mathcal G}$ and Proposition 3.7 from \cite{BG} (based on a result of Higson and Kasparov) it follows that the spectrum of the Cayley graph of $\tilde{\mathcal G}$ is $[-4,4]$.

Theorem \ref{ThMain} has many applications. Among them let us indicate an application to the theory of representations of branch and weakly branch groups. Branch  groups play  important  role  in many  investigations  in and around group theory  (see \eg \cite{Gr00}, \cite{BGS03} and \cite{Grig11}). The class of branch groups  contains  groups  of  intermediate  growth, amenable  but  not  elementary  amenable  groups,  groups  with  finite  commutator  width  \etc.   Weakly branch  groups are a natural generalization of  the  class  of  branch  groups playing  important  role in holomorphic  dynamics (see \cite{Nekr}) and  in the  theory  of  fractals  (see \cite{GNS15}). In Section \ref{SubsecAp} using Theorem \ref{ThMain} and results of \cite{DuGr15} and \cite{BG} we show (Corollary \ref{CoWeak}) that any subexponentially bounded weakly branch group $G$ admits uncountably many pairwise disjoint (not unitarily equivalent) pairwise weakly equivalent irreducible representations.

We finish the paper by presenting two examples of computation of spectra of Hecke type operators associated to the action of $\mathcal{G}$ on the boundary of a binary rooted tree. These examples illustrate the method of operator recursions used in  \cite{BG} and other places.


\section{Preliminaries.}
In this section we give necessary preliminaries. We deal with actions of countable groups on a standard probability space. A probability space is standard, if it is isomorphic modulo zero measure to an interval with Lebesgue measure, a finite or countable set of atoms, or a combination (disjoint union) of both. We refer the reader to \cite{Rokh} or \cite{Glas03} for details.

\subsection{Koopman and quasi-regular representations.}\label{SubsecKoopBranch}
A natural type of representations that one can associate to a measure-preserving action of a group $G$ on a measure space $(X,\mu)$, where $\mu$ is a quasi-invariant probability measure, is the Koopman representation $\kappa$ of $G$ in $L^2(X,\mu)$ acting by:$$(\kappa(g)f)(x)=\sqrt{\frac{\dd\mu(g^{-1}(x))}{\dd\mu(x)}}f(g^{-1}x),$$ where the expression under the radical is the Radon-Nikodym derivative. This representation is important due to the fact that the spectral properties of $\kappa$ reflect the dynamical properties of the action such as ergodicity and weak-mixing. One of the most natural questions concerning Koopman representations is whether it is irreducible.
 There are several examples of group actions with quasi-invariant measures known for which $\kappa$ is irreducible (see \eg \cite{BM11}, \cite{BC02}, \cite{CS91}, \cite{FTP83}, \cite{FTS94} and \cite{KuSt}), but typically this representation (or its "brother" $\kappa_0$) is not irreducible. In \cite{DuGr15} we constructed a new class of examples of irreducible Koopman representations arising from subexponential actions of weakly branch groups on boundaries of rooted trees.

 Recall that for $H<G$ a quasi-regular representation $\rho_{G/H}$ is a permutational representation of $G$ in $l^2(G/H)$ given by the natural action of $G$ on the set of left cosets $gH$. Given a countable group acting on a set $X$ and a point $x\in X$ one can define the quasi-regular representation $\rho_x$ in $l^2(Gx)$, where $Gx$ is the orbit of $x$,  by:
$$(\rho_x(g)f)(y)=f(g^{-1}y).$$ It is clear that $\rho_x$ is unitary equivalent to $\rho_{G/\st_G(x)}$, where $\st_G(x)$ is the stabilezer of $x$ in $G$. Notice that the isomorphism class of $\rho_x$ depends only on the stabilizer $\st_G(x)$ of $x$.

The question of irreducibility and disjointness of quasi-regular representations was studied by Mackey in \cite{Mack}. Using his criterion in \cite{BG} Bartholdi and the second author proved that for a weakly branch group $G$ and any $x\in\partial T$ the quasi-regular representation $\rho_x$ is irreducible. In addition, in \cite{DuGr15} the authors of the present paper showed that the representations $\rho_x$ associated to an action of a weakly branch group on the boundary of a rooted tree for $x\in\partial T$ from different orbits are pairwise disjoint (not unitary equivalent). In Section \ref{SubsecAp} we use Theorem \ref{ThMain} to strengthen this result for subexponentially bounded groups (Corollary \ref{CoWeak}).
\subsection{Groupoid representations and Hecke type operators.}\label{SubsecReps}
Here we briefly recall the construction of a  groupoid representation (see \cite{FM2} and \cite{Tak3} for details).
 As before, let $(X,\mu)$ be a standard probability space with a measure class preserving action of a countable group $G$ on it.
  Denote by $\mathcal{\mathcal{R}}$  the orbit equivalence relation on $X$. For $A\subset X^2$ and $x\in X$ set $A_x=A
 \cap (X\times\{x\}),A^x=A\cap (\{x\}\times X)$. Introduce measures $\nu_l,\nu_r$ on $\mathcal{\mathcal{R}}\subset X^2$ by
 $$\nu_l(A)=\int\limits_X |A^x|d\mu(x),\;\;\nu_r(A)=\int\limits_X |A_x|d\mu(x).$$ Notice
 that if $\mu$ is invariant with respect to $G$ then $\nu_l=\nu_r$. If $\mu$ is only quasi-invariant with respect to $G$ then the Radon-Nikodym derivative
 $D(x,y)=\tfrac{\dd\nu_l}{\dd\nu_r}(x,y)$ is defined and the relation
 $$\tfrac{\dd\mu(gx)}{\dd\mu(x)}=D(gx,x)\;\;\text{for each}\;\;g\in G\;\;\text{and $\mu$-almost all}\;\;x\in X$$ holds (see \cite{FM1}).
 The (left) groupoid representation of $G$ is the unitary representation $\pi$
  in $L^2(\mathcal{\mathcal{R}},\nu_r)$ defined by
$$(\pi(g)f)((x,y))=f(g^{-1}x,y).$$



   The next statement is of folklore type and is mentioned, for example, in \cite{Tak3}, \S 2.
 \begin{Prop}\label{Prop-grupp-equiv-int}
 The groupoid representation $\pi$ is unitarily equivalent to the representation $\int_X \rho_x d\mu(x)$.
\end{Prop}

Similarly to representation $\pi$ of $G$ in $L^2(\mathcal R,\nu_r)$ one can introduce a representation $\tilde\pi$ of $G$ in $L^2(\mathcal R,\nu_l)$ by
\begin{equation}\label{EqTilPi}(\tilde\pi(g)f)(x,y)=
\sqrt{\tfrac{\dd\mu(g^{-1}x)}{\dd\mu(x)}}f(g^{-1}x,y).\end{equation}
It is straightforward to verify that the representation $\tilde\pi$ is unitarily equivalent to $\pi$ via the intertwining isometry $\mathcal I:L^2(\mathcal R,\nu_r)\to L^2(\mathcal R,\nu_l)$ given by:
$$(\mathcal I f)(x,y)=\tfrac{1}{\sqrt{D(x,y)}} f(x,y).$$ The latter is well-defined since $D(x,y)\neq 0$ for  $\nu_r$-almost all $(x,y)\in\mathcal R$.


Let $U$ be a unitary representation of a countable group $G$ and $m\in\mathbb C[G]$, that is $m:G\to \mathbb C$ is a function of finite support. One can associate to $m$ a Hecke type operator \eqref{EqHecke}.  Additionally, given $\nu\in l^1(G)$ one can associate to it an operator $$U(\nu)=\sum\limits_{s\in G}\nu(s)U(s).$$ An interesting particular case is when $\nu$ is a measure on $G$ (\ie $\nu\in l^1(G)$ with $\nu(s)\geqslant 0$ for all $s\in G$).
For the case of quasi-regular representations spectral properties of these type operators are related to properties of random walks on graphs.


\subsection{Weak containment and spectrum of operators.}\label{SubsecWeak}
 Let $\rho$ and $\eta$ be two unitary representations of a group $G$ acting on Hilbert spaces $\mathcal H_\rho$ and $\mathcal H_\eta$ correspondingly. Then $\rho$ is weakly contained in $\eta$ (denoted by $\rho\prec\eta$) if for any $\epsilon>0$, any finite subset $S\subset G$ and any vector $v\in \mathcal H_\rho$ there exists a finite collection of vectors $w_1,\ldots,w_n\in\mathcal H_\eta$ such that  $$|(\rho(g)v,v)-\sum\limits_{i=1}^n(\eta(g)w_i,w_i)|<\epsilon$$ for all $g\in S$ (see \eg \cite{BeHaVa} for details).

In \cite{Dix69} Dixmier showed that for two unitary representations $\rho,\eta$ of a countable group $G$ one has $\rho\prec\eta$ if and only if $\|\rho(\nu)\|\leqslant \|\eta(\nu)\|$ for every $\nu\in l^1(G)$.
His result implies the following well known fact:
\begin{Prop}\label{CoWeakCond} Let $\rho,\eta$ be two unitary representations of a discrete group $G$. Then the following conditions are equivalent:
\begin{itemize}
\item[$1)$] $\rho\prec\eta$;
\item[$2)$] $\sigma(\rho(\nu))\subset\sigma(\eta(\nu))$ for all $\nu\in l^1(G)$;
\item[$3)$] $\|\rho(m)\|\leqslant \|\eta(m)\|$ for every positive $m\in \mathbb C[G]$.
\item[$4)$] there exists a surjective homomorphism $\phi:C_\eta\to C_\rho$ such that $\phi(\eta(g))=\rho(g)$ for all $g\in G$.
\end{itemize}
\end{Prop}\noindent
Here positiveness of an element $m$ of some $C^{*}$-algebra $A$ means that it can be represented as $m=x^{*}x$ with $x\in A$. Equivalently, $m$ is positive if it is self-adjoint ($m=m^{*}$) and $\sigma(m)\subset [0,\infty)$.

For an action of a countable group $G$ on a measure space $(X,\mu)$ denote by $\mathcal R=\mathcal R_{G,X}$ the equivalence relation generated by $G$ on $X$. In $1977$ Zimmer introduced a notion of amenability of ergodic action of $G$ on a measure space $(X,\mu)$ with a quasi-invariant probability measure $\mu$.  Later Adams, Eliott and Giordano \cite{AEG94} showed that Zimmer's amenability is equivalent to the following two conditions:

$1)$ $\mathcal R_{G,X}$ is $\mu$-hyperfinite (\ie on a set of full measure it is equal to a union of finite measurable equivalence relations);

$2)$ for $\mu$-almost all $x\in X$ the stabilizer $\st_G(x)$ is amenable.

Observe that condition $1)$ is equivalent to the following (see \eg \cite{FM1}, Proposition 4.1):

$1')$ on a set of full measure $\mathcal R_{G,X}$ coincides with $\mathcal R_{\mathbb Z,X}$ for some action of the group of integers $\mathbb Z$ on $(X,\mu)$ by measure class preserving transformations.


\begin{Th}[Kuhn]\label{ThKuhn} For an ergodic Zimmer amenable measure class preserving action of $G$ on a probability measure space $(X,\mu)$ one has $$\kappa\prec \lambda_G,$$ where $\kappa$ is the Koopman representation associated to the action of $G$ and $\lambda_G$ is the regular representation.
\end{Th}
\noindent At the end of Section \ref{SubsecKoopGroup} we will derive Kuhn's Theorem from part $2)$ of Theorem \ref{ThMain}. We refer the reader to \cite{AnDel03} for a generalization of Kuhn's Theorem for locally compact groups $G$. 

Another result related to Theorem \ref{ThMain}  is the following (see \cite{Pichot}, Theorem 30):
\begin{Th}[Pichot]\label{ThPichot} A measure class preserving action of a countable group $G$ on a standard probability space $(X,\mu)$ is hyperfinite if and only if for every $m\in l^1(G)$ with $\|m\|_1=1$ one has $\|\pi(m)\|=1$, where $\pi$ is the corresponding groupoid representation.
\end{Th}
\noindent Observe that the original result of Pichot concerns arbitrary discrete measured equivalence relations on $(X,\mu)$. However, all such equivalence relations are generated by group actions (see \cite{FM1}). Theorem \ref{ThPichot} is a reformulation of Pichot's result in terms of group actions. Theorem \ref{ThMain} implies the "only if" direction of Theorem \ref{ThPichot}.


The following result was the starting point of our investigation:
\begin{Th}[Bartholdi - Grigorchuk]\label{ThBG} Let $G$ be a finitely generated group acting on a regular rooted tree $T$ and $m\in\mathbb C[G]$. Then $\sigma(\rho_x(m))\subset \sigma(\kappa(m))$ for all $x\in\partial T$. If moreover the Schreier graph of the action of $G$ on the orbit $Gx$ of $x\in X$  is amenable, then $\sigma(\rho_x(m))=\sigma(\kappa(m))$.
\end{Th} \noindent For the proof of Theorem \ref{ThBG} we refer the reader to \cite{Grig11}, Proposition 10.4 (see also \cite{BG}, Theorem 3.6).



\section{Proof of Theorem \ref{ThMain}.}
We will split the proof of  Theorem \ref{ThMain} into three parts: Propositions \ref{PropMainHecke2}, \ref{Prop-spec-reg-Koop} and \ref{PropKoopGroup}.
\subsection{Equivalence of quasi-regular and groupoid representations.}\label{SubsecQRGr}
\begin{Prop}\label{PropMainHecke2} For an ergodic measure class preserving action of a countable group $G$ on a standard probability space $(X,\mu)$ one has $\rho_x\sim\pi$ for $\mu-$almost all $x\in X$.
\end{Prop}
\noindent The proof is based on a few technical statements. We will formulate these statements,
deliver Proposition \ref{PropMainHecke2} from them and then give the proofs of the statements.

For an action of a group $G$ on a space $X$, a subset $S=\{g_1,g_2,\ldots g_n\}\subset G, n\in\mathbb N$ and a point $x\in X$ introduce an \emph{orbital graph} $\Gamma_x=\Gamma_{x,g_1,\ldots,g_n}$ as a marked rooted graph whose vertex set is the set of points of the orbit $Gx$ ($x$ is the root) and such that $y,z\in Gx$ are connected by a directed edge marked by $g_i$ if and only if $z=g_iy$.
 Notice that here we don't assume that the group $G$ is generated by $S$, so the graphs $\Gamma_x$ are not necessary connected. In case if $S$ generates $G$ orbital graph $\Gamma_x$ coincide with marked Schreier graph defined by the triple $(G,\st_G(x),S)$. Fix a numeration of all elements of $G$:
 \begin{equation}\label{EqGNum} G=\{s_1,s_2,s_3,\ldots\}\;\;\text{with}\;\;s_1=e,\;\;\text{the unit element of}\;\;G.
 \end{equation}
 For $k\in \mathbb N$, $x\in X$ and $y\in Gx$ denote by $B_k(y)$ the subgraph of $\Gamma_x$ consisting of vertices $\{z\in Gx:z=s_iy\;\;\text{for some}\;\;i\leqslant k\}$ and all edges connecting them. We denote by $y$ the root of $B_k(y)$. Observe that $B_k(y)$ may be disconnected and that $$\Gamma_x=\bigcup\limits_{k\in\mathbb N}B_k(x).$$ 
 \begin{Def} We will say that two orbital graphs $\Gamma_x$ and $\Gamma_y$ are locally isomorphic if for any $k$ there exist a vertex $u$ of $\Gamma_x$ and a vertex $v$ of $\Gamma_y$ such that $B_k(u)$ is isomorhpic (as marked rooted graph) to $B_k(y)$ and $B_k(v)$ is isomorphic to $B_k(x)$.\end{Def}

The following statement is a straightforward modification of  Proposition 8.11 from \cite{Grig11}.
\begin{Prop}\label{PropLocIsom} Let $G$ act ergodically by measure class preserving transformations on a standard probability space $(X,\mu)$. Then there exists a subset $A\subset X$ of a full measure such that for any $g_1,\ldots g_n\in G,n\in \mathbb N$ and any $x,y\in A$ the marked orbital graphs $\Gamma_{x,g_1,\ldots,g_n}$ and $\Gamma_{y,g_1,\ldots,g_n}$ are locally isomorphic.
\end{Prop}
For $m\in\mathbb C[G]$ denote the support of $m$ by $$\supp(m)=\{g\in G:m(g)\neq 0\}.$$
\begin{Prop}\label{Prop-equiv-M} Let $G$ act on a space $X$, $x,y\in X$ and $m\in\mathbb C[G]$. Let $\supp(m)=\{g_1,g_2,\ldots,g_n\}$. If the orbital graphs $\Gamma_{x,g_1,\ldots,g_n}$ and $\Gamma_{y,g_1,\ldots,g_n}$  are locally isomorphic then $\sigma(\rho_x(m))=\sigma(\rho_y(m))$ .
\end{Prop}\noindent

The next Proposition is a standard statement about direct integral of Hilbert spaces. It can be derived from Lemma 2, \cite{Chow}. It is straightforward to see that all conditions of Lemma 2, \cite{Chow}, are satisfied in our case.
\begin{Prop}\label{Prop-int-M}
Let $(X,\mu)$ be a standard probability space, $\mathcal H=\int\limits_{X}\mathcal
H_xd\mu(x)$ be a direct integral of separable Hilbert spaces,
$M_x$ be an integrable family of operators and
$M=\int M_xd\mu(x)$. If the spectrum $\sigma(M_x)$ of almost
 all operators $M_x$ coincide and is equal to $\Sigma$ then $\sigma(M)=\Sigma$.
\end{Prop}
Now, let us derive Proposition \ref{PropMainHecke2} from the above statements.
\paragraph*{Proof of Proposition \ref{PropMainHecke2}.}  Let $m\in \mathbb C[G]$ and $\supp(m)=\{g_1,\ldots,g_n\}$. By Proposition \ref{PropLocIsom} for almost all $x$ the orbital graphs $\Gamma_{x,g_1,\ldots,g_n}$ are pairwise locally isomorphic. Proposition \ref{Prop-equiv-M} implies that the spectra $\sigma(\rho_x(m))$ coincide for almost all $x$. Denote this spectrum by $\Sigma$.
From Proposition \ref{Prop-int-M} we get that the spectrum of $$\int\limits_X \rho_x(m)\mathrm{d}\mu(x)$$
 is equal to $\Sigma$. From Proposition \ref{Prop-grupp-equiv-int} we get that
 $\sigma(\pi(m))=\Sigma$. Finally, Corollary \ref{CoWeakCond} implies that $\pi\sim\rho_x$ for almost all $x\in X$.


\paragraph*{Proof of Proposition \ref{PropLocIsom}.} Fix $n$ and $S=\{g_1,g_2,\ldots,g_n\}\subset G$. Let us call a finite rooted directed graph with edges marked by elements of $S$ $r$-admissible if it is isomorphic to $B_r(x)$ (as marked rooted graph) for some point $x\in X$. For an $r$-admissible graph $\Delta$ denote by $X_\Delta(r)$ the set of points $y\in X$ such that $B_r(y)$ is isomorphic to $\Delta$. For any fixed $r$ the sets $X_\Delta(r)$, where $\Delta$ is $r$-admissible, cover $X$, therefore, there exist $\Delta$ for which $X_\Delta(r)$ is of positive measure. We will call such $\Delta$ positively $r$-admissible. Let $P_r$ be the set of positively $r$-admissible graphs and $Z_r$ be the set of $r$-admissible but not positively $r$-admissible graphs. For an $r$-admissible graph $\Delta$ set
$$\tilde X_\Delta(r)=\bigcup\limits_{g\in G}g(X_\Delta(r)).$$ Clearly, for $\Delta\in P_r$ the set  $\tilde X_\Delta(r)$ is an invariant set of positive measure. Since the action is ergodic $\mu(\tilde X_\Delta(r))=1$. For $\Delta\in Z_r$ one has $\mu( X_\Delta(r))=0$. Denote $$X^S_*=\bigcap\limits_{r\geqslant 1}\bigcap\limits_{\Delta\in P_r}\tilde X_\Delta(r)\setminus \Big(\bigcup\limits_{r\geqslant 1}\bigcup\limits_{\Delta\in Z_r}X_\Delta(r)\Big).$$
Then $\mu(X_*^S)=1$.

 Further, let $x,y\in X^S_*$, $r\in\mathbb N$ and $\Delta=B_r(x)$. Definition of $X^S_*$ implies that $\Delta\in P_r$. Therefore, $y\in \tilde X_\Delta(r)$. Thus, $y\in g(X_\Delta(r))$ for some $g\in G$. This means that the marked rooted graphs $B_r(g^{-1}y),\Delta$ and $B_r(x)$ are pairwise isomorphic. We obtain that for any $x,y\in X^S_*$ the orbital graphs $\Gamma_{x,g_1,\ldots,g_n}$ and $\Gamma_{y,g_1,\ldots,g_n}$ are locally isomorphic.

Finally, denoting by $A$ the intersection of all sets of the form $X^S_*$ where $S$ is a finite subset of $G$ we obtain the desired.


\paragraph*{Proof of Proposition \ref{Prop-equiv-M}.} Let $G,X,x,y,m$ be as in the formulation of the Proposition \ref{Prop-equiv-M} and the orbital graphs $\Gamma_x=\Gamma_{x,g_1,\ldots,g_n}$ and $\Gamma_y=\Gamma_{y,g_1,\ldots,g_n}$ be locally isomorphic.
Set $$R=2\sum\limits_{g\in\supp(m)} |m(g)|.$$
Fix a point $\alpha$ from $\sigma(\rho_x(m))$ and let us show that $\alpha\in \sigma(\rho_y(m))$. Clearly,
$|\alpha|\leqslant\tfrac{1}{2}R$.
The proof of the following Lemma is straightforward and we omit it here.
\begin{Lm}\label{LmEquivNorm} Let $A$ be any bounded nonzero linear operator on a Hilbert space and $R\geqslant 2\|A\|$. Then
$$\alpha\in \sigma(A)\;\;\Leftrightarrow\;\;1\in \sigma(\mathrm{I}-
\tfrac{1}{R^2}(A-\alpha\mathrm{I})(A-\alpha\mathrm{I})^{*}),$$
where $\mathrm{I}$ is the identity operator.
\end{Lm}
Using Lemma \ref{LmEquivNorm} we obtain that for any unitary representation $\omega$ on $G$ one has:
$$\alpha\in \sigma(\omega(m))\;\;\Leftrightarrow\;\;1\in \sigma(\mathrm{I}-
\tfrac{1}{R^2}(\omega(m)-\alpha\mathrm{I})(\omega(m)-\alpha\mathrm{I})^{*}).$$ The
operator $\mathrm{I}-
\tfrac{1}{R^2}(\omega(m)-\alpha\mathrm{I})(\omega(m)-\alpha\mathrm{I})^{*}$
is of the form $\omega(s)$ for some $s\in \mathbb C[G]$,
 positive and of norm $\leqslant 1$. It follows that without loss of generality we may assume
that $\alpha=1$ and operators $\rho_x(m)$ and $\rho_y(m)$ are positive of norm $\leqslant 1$ (in fact,
$\|\rho_x(m)\|=1$, since we assume that $\alpha=1\in\sigma(\rho_x(m))$).

Further, consider orbital graphs $\Gamma_x$ and $\Gamma_y$. Let $\epsilon>0$. Since
$$\sup\limits_{\xi:\|\xi\|=1} (\rho_x(m)\xi,\xi)=1$$ we can find $l\in\mathbb N$ and a vector $\eta\in l^2(Gx)$
 supported on $B_l(x)$ such that $(\rho_x(m)\eta,\eta)>1-\epsilon$. Let
 $v\in\Gamma_y$ be such that $B_l(v)\subset \Gamma_y$ is isomorphic (as a rooted labeled graph) to
  $B_l(x)$. Let $\eta'\in l^2(B_l(v))\subset l^2(\Gamma_y)$ be a copy of $\eta$ via this isomorphism.
   Then one has:
  $$(\rho_x(m)\eta',\eta')=(\rho_y(m)\eta,\eta)>1-\epsilon.$$ It follows that $\|\rho_y(m)\|=1$
  and $1\in\sigma(\rho_y(m))$. This finishes the proof of Proposition \ref{Prop-equiv-M}.

\subsection{Weak containment of quasi-regular representations in the Koopman representation.}
\begin{Prop}\label{Prop-spec-reg-Koop}
$1)$ Let a countable group $G$ act on a standard probability space $(X,\mu)$, where $\mu$ is a quasi-invariant measure. Let $\kappa$ be the corresponding Koopman
representation in $L^2(X,\mu)$ and $\rho_x$ denotes the quasi-regular representation of $G$ in $l^2(Gx)$, $x\in X$. Then for almost all $x\in X$ one has $\rho_x\prec \kappa.$

\noindent
$2)$ If moreover $\mu$ is $G$-invariant and non-atomic then for almost all $x\in X$ one has $\rho_x\prec\kappa_0$, where $\kappa_0$ is the restriction of $\kappa$ onto the orthogonal complement to constant functions.
\end{Prop}
One of the ingredients of the proof is the following statement:
\begin{Lm}\label{Lm-A-refining} Let $T$ be a measure class preserving transformation of a standard probability space $(X,\mu)$ such that $Tx\neq x$ for almost all $x\in A$, where
$\mu(A)>0$. Then there exists $B\subset A,\mu(B)>0$ such that $\mu(B\cap TB)=0$.
\end{Lm} \noindent For the case of a measure preserving automorphism Lemma \ref{Lm-A-refining} follows from the proposition of $\S 1$ of \cite{Rokh49}. In fact, the same proof works in the case of a measure class preserving transformation. For the reader's convenience we provide here the arguments taken from \cite{Rokh49}, $\S 1$.
\begin{proof} Let us show first that there exists a measurable subset $C\subset A$ such that $\mu(T(C)\Delta C)\neq 0$. Fix a basis $\{A_i\}$ in $A$. Set
$$B_i=(A\setminus A_i)\cap T(A_i)\cup A_i\cap (A\setminus T(A_i)).$$ By definition of basis for almost all $x,y\in A$ such that $x\neq y$ there exists $A_i$ such that either $x\in A_i,y\in A\setminus A_i$ or $y\in A_i,x\in A\setminus A_i$. It follows that for almost all $x\in A$ there exists $i$ such that $x\in B_i$. Therefore, $\mu(\cup B_i)=\mu(A)>0$ and $\mu(B_i)>0$ for some i. Set $C=B_i$.

Now, if $\mu(C\setminus TC)\neq 0$ we set $B=C\setminus TC$. If $\mu(TC\setminus C)\neq 0$ we set $B=T^{-1}(TC\setminus C)$.
\end{proof}
\begin{Lm}\label{LmAkx}
Let $G$ act on $(X,\mu)$, where $\mu$ is a quasi-invariant probability measure. Let  $g_1,g_2,\ldots,g_n\in G$. For $x\in X$ and $k\in\mathbb N$ set
$$A_{k,x}:=\{y\in X:B_{k}(y)\;\;\text{is isomorphic to}\;\;B_{k}(x)\}.$$
 Then for almost all $x\in X$ one has $$\mu(A_{k,x})>0 \;\;\text{for all}\;\;k\in\mathbb N.$$\end{Lm}
\begin{proof} For every $k$ there are only finitely many distinct marked graphs appearing in the set $\{B_{2(k+1)}(x):x\in X\}$. Let $\mathcal B_k$ be the set of marked graphs $B$ such that $$\mu(\{x\in X:B_k(x)=B\})>0.$$  Consider $$M_k=\{x\in X:B_k(x)\in \mathcal B_k\}.$$ By construction, $\mu(M_k)=1$ and for every $x\in M_k$ one has: $\mu(A_{k,x})>0$. Let $$M=\bigcap\limits_{k\in\mathbb N}M_k.$$ Then $\mu(M)=1$ and for every $x\in M$ and every $k\in \mathbb N$ one has: $\mu(A_{k,x})>0$, which finishes the proof.
\end{proof}
\begin{proof}[{\bf Proof of Proposition \ref{Prop-spec-reg-Koop}}]
By Corollary \ref{CoWeakCond} it is sufficient to show that for all positive $m\in\mathbb C[G]$  for almost all $x\in X$ one has $\|\rho_x(m)\|\leqslant \|\kappa(m)\|$.
 Let $\supp(m)=\{g_1,\ldots,g_n\}$ and $A_{k,x}$ be the sets defined in Lemma \ref{LmAkx}.
 Till the end of the proof of this proposition fix $x$ such that $$\mu(A_{k,x})>0\;\;\text{for all}\;\;k\in\mathbb N.$$ Let $m\in\mathbb C[G]$ be a positive element. Without loss of generality we can assume that $\|\rho_x(m)\|=1$.
 Let $\epsilon>0$. Since
$$\sup\{(\rho_x(m)\xi,\xi):\xi\in l^2(\Gamma_x),\|\xi\|=1\}=1$$ we can find a unit vector $\eta\in l^2(\Gamma_x)$ of finite support such that $(\rho_x(m)\eta,\eta)>1-\epsilon$.

Further, fix $k$ such that $\supp(\eta)\subset B_k(x)$. Chose $K\in\mathbb N$ such that $gB_k(x)\subset B_K(x)$ for all $g\in\supp(m)$. 
Observe that for every $y\in A_{K,x}$, any $i,j\leqslant k$ and any $g,h\in\supp(m)$ one has: $$gs_iy=hs_jy\;\;\Leftrightarrow\;\;
gs_ix=hs_jx.$$ 
Using Lemma \ref{Lm-A-refining} successively for all elements of the form $s_j^{-1}h^{-1}gs_i,$ where $i,j\leqslant k$ and $g,h\in\supp(m)$ such that
$s_j^{-1}h^{-1}gs_ix\neq x$ we can find $B\subset A_{k,x}$ such that $\mu(B)>0$ and
$$\mu(gs_iB\cap hs_jB)=0\;\;\text{for all such}\;\;g,h,s_i,s_j.$$

Further, divide the set of positive numbers $\mathbb R_+$ into subintervals
$$I_s=[(1+\epsilon)^s,(1+\epsilon)^{s+1}),s\in\mathbb Z$$ so that for every $s\in\mathbb Z$ one has $ab^{-1}\in (1-\epsilon,1+\epsilon)$ for $a,b\in I_s$. For every function $f:\{1,\ldots,k\}\times\supp(m)\to\mathbb Z$ introduce the set
$$B_f=\{t\in B:\sqrt{\tfrac{\dd \mu(gs_it)}{\dd \mu(s_it)}}\in I_{f(i,g)}\;\;\text{for every}\;\;g\in \supp(m),1\leqslant i\leqslant k\}.$$ Since union of the sets $B_f$ over all function $f$ is the set $B$ of positive measure one has $\mu(B_f)>0$ for some $f$. Fix such $f$. Then for every $g\in \supp(m),1\leqslant i\leqslant k$ we have:
$$\sqrt{\tfrac{\mu(gs_iB_f)}{\mu(s_iB_f)}}\in I_{f(i,g)}\;\;\text{and}\;\;\Big|1-\sqrt{\tfrac{\dd \mu(t)}{\dd \mu(gt)}}\sqrt{\tfrac{\mu(gs_iB_f)}{\mu(s_iB_f)}}\Big|<\epsilon\;\;\text{for all}\;\;t\in s_iB_f.$$

Finally, for $1\leqslant i\leqslant k,g\in\supp(m)$ and $y=s_ix$ consider the function  $e_y=\frac{1}{\sqrt{\mu(gs_iB_f)}}\mathbbm{1}_{gs_iB_f}$. Observe that
\begin{equation}\label{Eqey}\|e_y-\kappa(s_i)e_x\|^2=\int\limits_{s_iB_f}\Big(\tfrac{1}{\sqrt{\mu(s_iB_f)}}-
\tfrac{1}{\sqrt{\mu(B_f)}}\sqrt{\tfrac{\dd \mu(s_i^{-1}t)}{\dd \mu(t)}}\Big)^2\dd\mu(t)<\epsilon.\end{equation}
Consider the spaces \begin{align*}\mathcal H_x=\Sp\{\delta_y:y=gs_ix,i\leqslant k,g\in\supp(m) \}\subset l^2(Gx),\\
\mathcal L_x=\Sp\{e_y:y=gs_ix,i\leqslant k,g\in\supp(m)\}\subset L^2(X,\mu).\end{align*} The map $\phi:\delta_y\to e_y$ induces an isometry between these spaces. Moreover, by inequality \eqref{Eqey} for every $h\in \supp(m)$ and every $y=s_ix\in B_k(x)$, where $i\leqslant k$, we have
\begin{align*}\|\phi(\rho_x(h)\delta_y)-\kappa(h)e_y\|^2=\|e_{hy}-\kappa(h)e_y\|^2\\
\leqslant (\|e_{hs_ix}-\kappa(hs_i)e_x\|+\|\kappa(h)(e_{s_ix}-\kappa(s_i)e_x)\|)^2\leqslant 4\epsilon.\end{align*} This implies that $$\|\phi(\rho_x(h)\eta)-\kappa(h)\phi(\eta)\|\leqslant 2\sqrt\epsilon$$ for every $h\in\supp(m)$ and thus
\begin{equation}\label{EqKaRhoIneq}\|\phi(\rho_x(m)\eta)-\kappa(m)\phi(\eta)\|\leqslant 2\sqrt\epsilon\|m\|_1,\end{equation}
where $\|m\|_1=\sum\limits_{h\in\supp(m)}|m(h)|$. Since $\|\rho_x(m)\eta\|\geqslant 1-\epsilon$ for arbitrary $\epsilon>0$ we obtain that $\|\kappa(m)\|=1$ and $1\in \sigma(\kappa(m))$. This finishes the proof of part $1)$ of Proposition \ref{Prop-spec-reg-Koop}.

Now let $\mu$ be $G$-invariant and non-atomic. From ergodicity it follows that the orbit $Gy$ is infinite for almost all $y\in X$. Without loss of generality we can assume that $Gx$ is infinite. Fix $\epsilon>0$ and  a unit vector $\eta\in l^2(\Gamma_x)$ of finite support such that $(\rho_x(m)\eta,\eta)>1-\epsilon$. Assume that $$\alpha=\sum_{y\in\supp(\eta)}\eta(y)\neq 0.$$ Then choose arbitrarily a sequence of distinct elements $y_i$ from $Gx\setminus \supp(\eta)$ and for $n\in\mathbb N$ introduce
$$\eta_n=\eta-\tfrac{\alpha}{n}\sum\limits_{i=1}^n\delta_{y_i}\in l^2(\Gamma_x),
\;\;m_n=m+\tfrac{1}{n^2}\sum\limits_{i=1}^n\delta_{h_i}\in \mathbb C[G],$$ where $h_i\in G$ are such that $h_ix=y_i$. Clearly,
$$\sum_{y\in\supp(\eta_n)}\eta_n(y)=0,\;\;\text{and}\;\;\lim\limits_{n\to\infty}\eta_n=\eta\;\;\text{in}\;\;l^2\text{-norm}.$$ Moreover,
$\lim\limits_{n\to\infty}\rho_x(m_n)=\rho_x(m)\;\;\text{and}\;\;\lim\limits_{n\to\infty}\kappa(m_n)=\kappa(m)$ where the limits are in the strong operator topology. Therefore, without loss of generality we may assume that
$$\sum_{y\in\supp(\eta)}\eta(y)=0.$$ Then by construction $\phi(\eta)\in L^2(X,\mu)$ is orthogonal to constant functions, and thus the representation $\kappa$ in the inequality \eqref{EqKaRhoIneq} can be replaced by $\kappa_0$. When $\epsilon\to 0$ we obtain that $\|\kappa_0(m)\|=1$. This finishes the proof of part $2)$ of Proposition \ref{Prop-spec-reg-Koop} and hence finishes the proof of parts $1)$ and $2)$ of Theorem \ref{ThMain}.
\end{proof}
\begin{Rem} The condition of non-atomicity of measure $\mu$ in the second part of Proposition \ref{Prop-spec-reg-Koop} is necessary. Consider $G=\mathbb Z_2=\{0,1\}$. Equip $X=\mathbb Z_2$ with the uniform probability measure $\mu$. Let $G$ act on $(X,\mu)$ by shifts. Then for any $m=\alpha\delta_0+\beta\delta_1\in\mathbb C[G]$ one has:
$$\sigma(\rho_0(m))=\sigma(\rho_1(m))=\{1,\alpha-\beta\},\;\;\sigma(\kappa_0)(m)=\{\alpha-\beta\}$$ and thus the two spectra do not coincide when $\alpha-\beta\neq 1$.
\end{Rem}
\subsection{Equivalence of Koopman and groupoid representations for a hyperfinite action.}\label{SubsecKoopGroup}
Part $3)$ of Theorem \ref{ThMain} follows from the next:
\begin{Prop}\label{PropKoopGroup}
For a hyperfinite measure class preserving action of a countable group $G$ on a standard probability space $(X,\mu)$ one has $\kappa\sim\pi.$
\end{Prop}
\noindent In the proof we will use the following result (see \cite{ChacFried:65}, Lemma 4):
\begin{Lm}\label{ThRokh} Let $U$ be an aperiodic measure class preserving transformation of a Lebesgue space $(X,\mu)$. Then for any $N$ and any $\epsilon>0$ there exists a  measurable set $A$ such that the sets $A,UA,\ldots,U^{N-1}A$ are pairwise disjoint and $\mu(A\cup UA\cup\ldots\cup U^{N-1}A)>1-\epsilon$. \end{Lm}\noindent
Lemma \ref{ThRokh} is a generalization of the famous Rohlin Lemma from \cite{Rokh49} to the case of quasi-invariant measures.
\begin{proof}[{\bf Proof of Proposition \ref{PropKoopGroup}}] First notice that in the case of a finite $X$ the groupoid representation $\pi$ is unitarily equivalent to a direct sum of finitely many copies of the Koopman representation $\kappa$. Therefore, without loss of generality we can assume that for all $x\in X$ the orbit $Gx$ is infinite. Since Koopman representation uses a Radon-Nikodym derivative in the definition it will be convenient to replace $\pi$ by a unitarily equivalent representation $\tilde \pi$ (see \eqref{EqTilPi}). By Corollary \ref{CoWeakCond} and part $1)$ of Theorem \ref{ThMain} it is sufficient to show that
$\|\kappa(m)\|\leqslant\|\tilde\pi(m)\|$ for every positive element $m\in\mathbb C[G]$.
Without loss of generality we will assume that $\|\kappa(m)\|=1$.

Since the action of $G$ on $X$ is
hyperfinite there exists a measure-class preserving automorphism $U$ generating the equivalence relation $\mathcal R$ generated by $G$ on $X$ (see \eg \cite{FM1}, Proposition 4.1). Clearly, $U$ is aperiodic. For $g\in G$ and $x\in X$ denote by $n_g(x)$ the integer number such that
$gx=U^{n_g(x)}x$. Observe that for every $g\in G$ the function $n_g(x)$ is measurable.
Fix $\delta>0$. Let $\eta\in L^2(X,\mu)$ be a unit vector such that $(\kappa(m)\eta,\eta)>1-\delta$.
Without loss of generality we may assume that the set of values
of $\eta$ is finite. Let $K=\max\{\|\eta\|_\infty,1\}$.

Further, find a number $L$ such that
$$\mu(\{x:|n_{g^{\pm 1}}(x)|\leqslant L\;\text{for all}\;g\in\supp(m)\})\geqslant 1-\tfrac{\delta}{2K^2}.$$ Introduce a set $$\Omega=\{x:|n_{g^{\pm 1}}(x)|\leqslant L\;\text{for all}\;g\in\supp(m)\}.$$ Choose $N$ such that $\tfrac{L}{N}\leqslant\tfrac{\delta}{8K^2}$. Using Lemma \ref{ThRokh} one can construct a set $C$ such that the sets $C,UC,\ldots, U^{N-1}C$ are pairwise disjoint,
and $$\mu(C\cup  UC\cup\ldots\cup U^{N-1}C)\geqslant 1-\tfrac{\delta}{4K^2}.$$ Set $C_j=U^j(C)$ for $j=0,1,\ldots,N-1$. Let $\Sigma=(C_L\cup C_{L+1}\cup\ldots\cup C_{N-L-1})\cap \Omega$. Then $\mu(\Sigma)\geqslant 1-\tfrac{\delta}{K^2}$.
Consider the functions
$$\tilde\eta(x,y)=\eta(x)\mathbbm{1}_C(y)\sum\limits_{j=0}^{N-1}\delta_{x,U^j(y)},
$$ where $\delta_{x,y}$ is the Kronecker delta,
$$\tilde\eta_0(x,y)=\mathbbm{1}_\Sigma(x)\tilde\eta(x,y)
\;\;\text{and}\;\;\eta_0(x)=\mathbbm{1}_\Sigma(x)\eta(x),$$ where $\mathbbm{1}_A$ stands for the characteristic function of a set $A$.
Observe that for every $x\in X$ there exists at most one $y$ such that $\tilde\eta(x,y)\neq 0$.
By definition of $\nu_l$ one has:
$$\|\tilde\eta\|^2=\int\limits_X\sum\limits_{y\sim x}|\tilde\eta(x,y)|^2\dd\mu(x)=\sum\limits_{j=0}^{N-1}\int\limits_{C_j}|\eta(x)|^2\dd\mu(x)\leqslant \|\eta\|^2.$$ Using similar computations one can show that $\|\tilde\eta_0\|=\|\eta_0\|$. Since $\mu(\Sigma)>1-\delta$ we obtain that $$\|\eta\|^2-\delta\leqslant \|\tilde\eta_0\|^2\leqslant \|\tilde\eta\|^2\leqslant \|\eta\|^2.$$

Let $g\in\supp(m)$. Assume that $x\in C_j\cap A$, where $L\leqslant j\leqslant N-L-1$.
Let $y=U^{-j}(x)$. One has $-L\leqslant n_{g^{-1}}(x)\leqslant L$ and $g^{-1}x\in C_{j+n_{g^{-1}}(x)}$.
It follows that $$\tilde\eta_0(x,y)=\eta_0(x),\;\;\tilde\eta(g^{-1}x,y)=\eta(g^{-1}x).$$
If $x\notin \Sigma$ then $\eta_0(x)=0$ and $\tilde\eta_0(x,y)=0$ for all $y\in Gx$. We obtain:
\begin{align*}(\tilde\pi(g)\tilde\eta,\tilde\eta_0)=\int\limits_X\sum\limits_{y\sim x}\sqrt{\tfrac{\dd\mu(g^{-1}(x))}{\dd\mu(x)}}\tilde\eta(g^{-1}x,y)\overline{\tilde\eta_0(x,y)}\dd\mu(x)=\\
\int\limits_X\sqrt{\tfrac{\dd\mu(g^{-1}(x))}{\dd\mu(x)}}\eta(g^{-1}x)\eta_0(x)\dd\mu(x)=(\kappa(g)\eta,\eta_0).\end{align*}
Since $\|\tilde\eta-\tilde\eta_0\|\leqslant\|\eta-\eta_0\|\leqslant \delta^{\frac{1}{2}}$ the latter implies that $$|(\pi(g)\tilde\eta,\tilde\eta)-(\kappa(g)\eta,\eta)|\leqslant 2\delta^{\frac{1}{2}}.$$
Finally, we get: $$|(\tilde\pi(m)\tilde\eta,\tilde\eta)-(\kappa(m)\eta,\eta)|\leqslant 2\delta^{\frac{1}{2}}\|m\|_1,\;\;\text{where}\;\;\|m\|_1=\sum\limits_{g\in\supp(m)} |m(g)|.$$ Since $\delta>0$ is arbitrary, the inequality $\|\pi(m)\|\geqslant\|\kappa(m)\|$ follows. This finishes the proof of Proposition \ref{PropKoopGroup} and Hence part $3)$ of Theorem \ref{ThMain}.\end{proof}

Theorem 1 is now proven. We finish this section by deriving Kuhn's Theorem \ref{ThKuhn} from Theorem \ref{ThMain}.
\begin{proof}[{\bf Proof of Theorem \ref{ThKuhn}}] Let $G$ act ergodically by measure class preserving automorphisms on a probability measure space $(X,\mu)$. Assume that this action is Zimmer amenable.  Then by Theorem \ref{ThMain} the corresponding Koopman representation is weakly equivalent to the quasi-regular representation $\rho_x$ of $G$ for almost every $x$. By result of Adams, Eliott and Giordano \cite{AEG94} Zimmer's amenability implies that $\st_G(x)$ is amenable for almost every $x$. Let $x$ be such that $\rho_x\sim\kappa$ and $\st_G(x)$ is amenable. Using the well known fact that
for a subgroup $H<G$ $$\rho_{G/H}\prec\lambda_G\;\;\text{if and only if}\;\;H\;\;\text{is amenable}$$ (see \eg \cite{BG}, Proposition 3.5) we obtain that  $\rho_x\prec\lambda_G$. This shows that $\kappa\prec \lambda_G$.
\end{proof}

\section{Applications to weakly branch groups.}\label{SubsecAp}
 We recall some notions related to group actions on rooted trees. We refer the reader to \cite{Grig11} and \cite{GNS00} for detailed definitions and properties of these actions.

A $d$-regular rooted tree is a tree $T$, with vertex set divided into levels $V_n$, $n\in\mathbb Z_+$, such that $V_0$ consists of one vertex $v_0$ (called the root of $T$), the edges are only between consecutive levels, and each vertex from $V_n$, $n\geqslant 0$ (we consider infinite trees), is connected by an edge to exactly $d$ vertices from $V_{n+1}$ (and one vertex from $V_{n-1}$ for $n\geqslant 1$). An automorphism of a rooted tree $T$ is any automorphism of the graph $T$ preserving the root. Denote by $\aut(T)$ the group of automorphisms of $T$.

Let $T$ be a $d$-regular rooted tree, $d\geqslant 2$, and $G<\aut(T)$. The rigid stabilizer of a vertex $v$  is the subgroup $\rist_v(G)=\{g\in G:\supp(g)\subset T_v\}$. The rigid stabilizer of level $n$ is
$$\rist_n(G)=\prod\limits_{v\in V_n}\rist_v(G).$$ $G$ is called \emph{branch} if it is transitive on each level and $\rist_n(G)$ is a subgroup of finite index in $G$ for all $n$. $G$ is called \emph{weakly branch} if it is transitive on each level $V_n$ of $T$ and $\rist_v(G)$ is nontrivial for each $v$.

For each level $V_n$ of a $d$-regular rooted tree an automorphism $g$ of $T$ can be presented in the form
\begin{equation}\label{EqRest}g=\sigma\cdot(g_1,\ldots,g_{d^n}),\end{equation} where $\sigma\in\sym(V_n)$ is a permutation of the vertices from $V_n$ and $g_i$ are the restrictions of $g$ on the subtrees emerging from the vertices of $V_n$.

 For an element $g\in\aut(T)$ denote by $k_n(g)$ the number of restrictions $g_i$ to the vertices of level $n$ such that $g_i$ is not equal to the identity automorphism. We call $g$ subexponentially bounded if for every $0<\gamma<1$ one has
 $$\lim\limits_{n\to\infty}k_n(g)\gamma^n=0.$$ A group $G<\aut(T)$ is subexponentially bounded if each $g\in G$ is subexponentially bounded. Many important examples of branch and weakly branch groups (\eg the group $\mathcal G$ of intermediate growth constructed by the second author, Gupta-Sidki $p$-groups and Basilica group) are subexponentially bounded groups.

For a $d$-regular rooted tree $T$ its boundary $\partial T$ is the set of infinite paths starting from $v_0$. Observe that $\partial T$ can be identified with a space of sequences $\{x_j\}_{j\in\mathbb N}$ where $x_j\in\{1,\ldots,d\}$.  For a vertex $v$ of $T$ we denote by $\partial T_v\subset\partial T$ the set of paths passing through $v$. Supply $\partial T$ by the topology generated by the sets $\partial T_v$. Automorphisms of $T$ act naturally on $\partial T$ by homeomorphisms. Notice that $\partial T$ admits a unique $\aut(T)$-invariant measure $\mu$. This measure is uniform in the sense that
 $$\mu(\partial T_v)=\tfrac{1}{d^n}\;\;\text{for any}\;\;n\;\;\text{and any}\;\;v\in V_n.$$ In \cite{GNS00} it is shown that this measure is ergodic with respect to a group $G<\aut(T)$ if and only if the action of $G$ is transitive on each level $V_n$ of $T$. Moreover, in this case it is uniquely ergodic.

 Further, let $G$ be a weakly branch group acting on a $d$-regular rooted tree $T$. Recall that for $x,y\in \partial T$ from the same $G$-orbit the corresponding quasi-regular representations are unitary isomorphic. Denote by $\mathcal O$ the set of orbits of $G$ on $\partial T$. For $\omega\in\mathcal O$ denote by $\rho_\omega$ the corresponding quasi-regular representation of $G$. Let \begin{equation}\label{EqMP}\mathcal P=\{p=(p_1,p_2,\ldots,p_d):p_i> 0\;\;\text{for}\;\;i=1,2,\ldots,d\;\;\text{and}\;\;\sum\limits_{i=1}^dp_i=1\}\end{equation} be the set of all probability distributions on the alphabet $\{1,2,\ldots,d\}$ assigning positive probability to every letter and \begin{equation}\label{EqMP*}\mathcal P^{*}=\{p\in\mathcal P:p_i\neq p_j\;\;\text{for all}\;\;1\leqslant i< j\leqslant d\}.\end{equation} For $p\in\mathcal P^{*}$ denote by $\mu_p=\prod\limits_{\mathbb N}p$ the corresponding Bernoulli measure on $\partial T$. It is shown in \cite{DuGr15}, Proposition 2,  that subexponentially bounded automorphisms preserve the measure class of $\mu_p$. Assuming that $G$ is a subexponentially bounded group we denote by $\kappa_p$ the Koopman representation associated to the action of $G$ on $(\partial T,\mu_p)$.

Using Mackey's criterion of irreducibility of quasi-regular representations Bartholdi and Grigorchuk in \cite{BG} showed that quasi-regular representations $\rho_\omega$ corresponding to the action of a weakly branch group $G$ on the boundary of a rooted tree are irreducible for all $\omega\in \mathcal O$. Moreover, in \cite{DuGr15} the authors proved the following.
 \begin{Th} Let $G$ be a subexponentially bounded weakly branch group acting on the boundary of a $d$-regular rooted tree $T$. For every $p\in\mathcal P^{*}$ the representation $\kappa_p$ of $G$ is irreducible. Moreover, the representations of $G$ from  $\{\kappa_p:p\in\mathcal P^{*}\}\cup \{\rho_\omega:\omega\in\mathcal O\}$ are pairwise disjoint.\end{Th}\noindent
Using Theorem \ref{ThMain} we will strengthen this result (Corollary \ref{CoWeak}).

  For $g\in \aut(T)$ a point $x=x_1x_2x_3\ldots\in\partial T$  is called $g$-rigid if there exist $n\in\mathbb N$ and $v=x_1x_2\ldots x_n\in V_n$ such that the restriction $g|_{\partial T_v}$ is trivial. For $G<\aut(T)$ denote by $R(G)$ the set of points $x\in \partial T$ such that $x$ is $g$-rigid for all $g\in G$. Such points are called rigid. Let $\mathcal O_{R(G)}$ be the set of $G$-orbits of points from $R(G)$. From the proof of Proposition 2 of \cite{DuGr15} we obtain:
 \begin{Lm}\label{LmSubexpRigid} Let $T$ be a $d$-regular rooted tree and $G<\aut(T)$ be subexponentially bounded. Then for any $p\in\mathcal P$ from \eqref{EqMP} one has $\mu_\mathcal{P}(R(G))=1$.\end{Lm} Recall that for an action of a group $G$ by homeomorphisms on a topological space $X$ a point $x$ is called \emph{typical} if for every $g\in G$ either $gx\neq x$ or $g$ acts trivially on some neighborhood of $x$. Clearly, the set of all typical points is open and $G$-invariant. Observe that for $G<\aut(T)$ rigid point $x\in\partial T$ is typical. The next proposition is a topological version of Proposition \ref{PropLocIsom} 
 and is a generalization of Proposition 6.21 from \cite{GNS00} (see also \cite{Grig11}, Proposition 8.8).
 \begin{Prop}
 \label{PropRegIsom} Let $G$ be a countable group acting minimally on a topological space $X$. Let $n\in\mathbb N$ and $g_1,\ldots,g_n\in G$. Then for any typical points $x,y\in X$ the orbital graphs $\Gamma_{x,g_1,\ldots,g_n}$ and $\Gamma_{y,g_1,\ldots,g_n}$ are locally isomorphic.
 \end{Prop}
\begin{proof} Let $A$ be the set of all typical points. Fix $r\in\mathbb N$. Denote by $D_r$ the set of finite rooted marked graphs $\Delta$ such that there exists $x\in A$ for which $B_r(x)=\Delta$. For any $\Delta\in D_r$ denote by $X_\Delta(r)$ the set of $x\in A$ such that $B_r(x)=\Delta$. Given a point $x\in A$, $1\leqslant l,m\leqslant r$ and $1\leqslant i\leqslant n$ by definition of a typical point there exists a neighborhood $U(x)$ such that either \vskip 0.1cm
\noindent $a)$ for all $y\in U(x)$ one has $s_m^{-1}g_is_ly=y$ (\ie the vertices $s_ly$ and $s_my$ are connected by an edge marked by $g_i$ in $\Gamma_y$) or \vskip 0.1cm
\noindent $b)$ for all $y\in U(x)$ one has $s_m^{-1}g_is_ly\neq y$ (\ie the vertices $s_ly$ and $s_my$ are not connected by an edge marked by $g_i$ in $\Gamma_y$).\vskip 0.1cm
 \noindent It follows that the sets $X_\Delta(r)$ are open.

 Further, let $x,y\in A$ and $\Delta=B_r(x)$. By minimality of the action of $G$ on $X$ there exists $g\in G$ such that $gy\in X_\Delta$. Thus, $B_r(x)=\Delta=B_r(gy)$ which finishes the proof.
\end{proof}

\noindent Combining these results with Theorem \ref{ThMain} and taking into account that any subexponentially bounded weakly branch group generates a hyperfinite equivalence relation on $\partial T$ (see \cite{GrigNekr:Amen}, Theorem of Section 3) we obtain:
\begin{Co}\label{CoWeak} For any subexponentially bounded weakly branch group $G$ acting on a $d$-regular rooted tree ($d\geqslant 2$) the representations from
$\{\kappa_p:p\in\mathcal P^{*}\}\cup \{\rho_\omega:\omega\in\mathcal O_{R(G)}\}$
are irreducible, pairwise disjoint (not unitarily equivalent), and pairwise weakly equivalent.
\end{Co}

Observe that $\mathcal P^{*}$ and $R(G)$ have cardinality of continuum. Similar to Corollary \ref{CoWeak} results are known for free groups $F_n$, $n\geqslant 2$. Namely, for every $n\geqslant 2$ there exists a continuum of irreducible pairwise disjoint and pairwise weakly equivalent Koopman type representations of $F_n$ (see \eg \cite{KuSt}). However, weak equivalence of these representations uses the fact that the reduced $C^{*}$-algebra of $F_n$ is simple (see \cite{Pow75}). The class of weakly branch groups contains many amenable groups. Recall that groups of intermediate growth are amenable but not elementary amenable. For any amenable group the reduced $C^{*}$-algebra is not simple (see \eg \cite{harpe07}). To the authors' best knowledge Corollary \ref{CoWeak} gives the first example of amenable groups admitting a continuum of pairwise disjoint weakly equivalent irreducible representations. Notice that all $p$-groups of intermediate growth constructed in \cite{Gr84} and \cite{Gr85} (for each prime $p$ there are  $2^{\chi_0}$ such groups) as well as Gupta-Sidki $p$-groups are  bounded (and therefore subexponentially bounded) amenable branch groups and hence satisfy the conditions of Corollary 3.

\section{Examples and proof of Theorem \ref{ThGammaSpec}.}
One of the basic examples of branch groups is the group $\mathcal{G}$ mentioned in the introduction. This group acts on the boundary of the binary rooted tree (which we will denote by $T$) and is generated by elements $a,b,c,d$ satisfying the following recursions: \begin{equation}\label{EqGrigRec}a=\sigma\cdot(\id,\id),\;b=(a,c),\;c=(a,d),\;d=(\id,b),\end{equation} where $\id$ is the identity action (see \eg \cite{Gr00}, \cite{Grig11} or \cite{GNS00}).

Set \begin{equation}\label{EqDel}\Delta=\tfrac{1}{4}(a+b+c+d)\in \mathbb C[\mathcal{G}].\end{equation} Let $\kappa$ be the Koopman representation corresponding to the action of $\mathcal{G}$ on $(\partial T,\mu)$, where $\mu$ is the unique probability $\mathcal G$-invariant measure on $\partial T$ (\ie $(\tfrac{1}{2},\tfrac{1}{2})$ uniform Bernoulli measure on $\partial T=\{0,1\}^\mathbb N$). In \cite{BG} Bartholdi and the second author developed a method for calculating spectra of Hecke type operators associated with self-similar groups and showed the following:
\begin{Th}\label{ThBGDelta} For all $x\in\partial T$ one has $$\sigma(\kappa(\Delta))=\sigma(\rho_x(\Delta))=[-\tfrac{1}{2},0]\cup[\tfrac{1}{2},1]\subset\sigma(\lambda_\mathcal{G}(\Delta)).$$
\end{Th}
\noindent
Also in \cite{BG} spectra of Hecke type operators for other groups are calculated. The main tools authors used were operator recursions based on relations of type \eqref{EqGrigRec}, Schur complement and the reduction of the spectral problem to the problem of finding a suitable invariant set for the associated rational map $\mathbb R^n\to\mathbb R^n$ for some $n$.

Notice that the spectrum of $\rho_x(\Delta)$ coincides with the spectrum of the Schreier graph $\Gamma_x$ for every $x\in\partial T$. An important characteristic of the Schreier graph $\Gamma_x$, $x\in\partial T$, is the spectral measure of $\rho_x(\Delta)$ associated to the unit vector $\delta_x\in l^2(Gx)$. For the action of $\mathcal G$ on $\partial T$ these measures were computed in \cite{GrigKryl}.

In this section we compute the spectrum of the Cayley graph of $\mathcal G$ thus proving Theorem \ref{ThGammaSpec}. Also, using operator recursions similar to those from \cite{BG} we prove the results analogous to Theorem \ref{ThBGDelta} for the Koopman representations $\kappa_{(q,1-q)}$ of $\mathcal{G}$, $0<q<1$, and for the groupoid representation of $\mathcal{G}$ corresponding to the invariant Bernoulli measure $\mu$ on $\partial T$. Surprisingly, the spectrum does not depend on the parameter $q$ defining the measure.

\subsection{Spectrum of the Cayley graph of $\mathcal{G}$.}
Here we prove Theorem \ref{ThGammaSpec} which is equivalent to: $$\sigma(\lambda_\mathcal{G}(\Delta))=[-\tfrac{1}{2},0]\cup[\tfrac{1}{2},1],$$ where $\lambda_\mathcal{G}$ is the regular representation of $\mathcal G$.

Introduce a 2-parameter family of elements $Q(\alpha,\beta)=4\Delta-(\alpha+1)a-(\beta+1)e=-\alpha a+b+c+d-(\beta+1)e\in\mathbb C[\mathcal{G}]$, where $e$ is the identity element of $\mathcal{G}$. For a unitary representation $\rho$ of $\mathcal{G}$ let
$\Sigma_\rho$ be the set of pairs $(\alpha,\beta)\in\mathbb R^2$ such that $\rho(Q(\alpha,\beta))$ is not invertible.\begin{figure}\centering\includegraphics[width=0.4\textwidth]{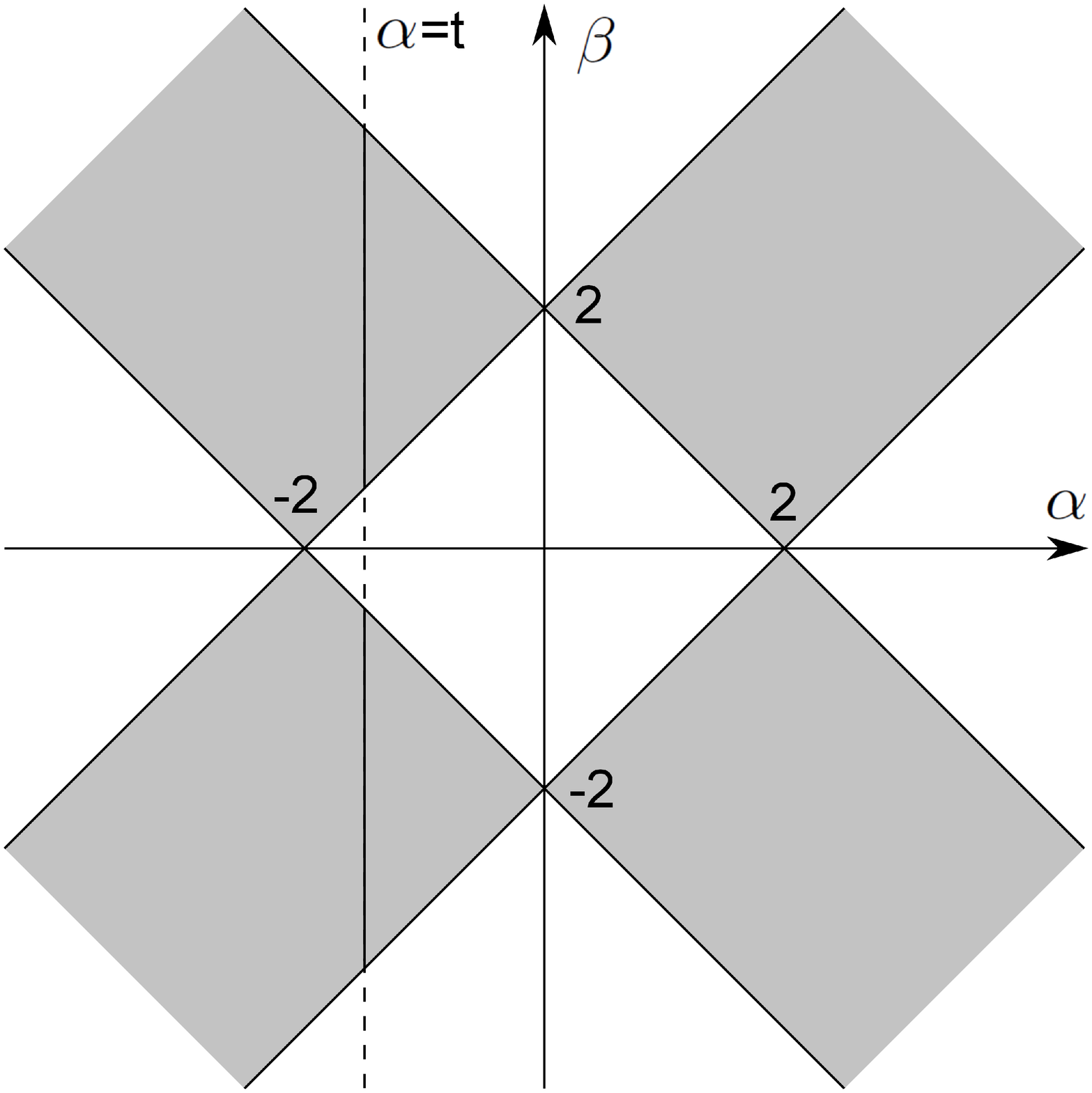}\caption{The set $\Omega$.}\label{FigSigma}\end{figure} Let $\Omega=\{(\alpha,\beta):||\alpha|-|\beta||\leqslant 2,\;||\alpha|+|\beta||\geqslant 2\}$ (see figure \ref{FigSigma}).
\begin{Lm}\label{LmSigmaAny} For any unitary representation $\rho$ one has: $\Sigma_\rho\subset\Omega$.
 \end{Lm}
 \begin{proof}  Using the basic relation in $\mathcal G$ it is straightforward to verify that $(b+c+d-e)^2=4e\in\mathbb C[\mathcal{G}]$, where $e\in\mathcal{G}$ is the group unit. Let $A=\rho(a)$, $U=\rho(\tfrac{1}{2}(b+c+d-e))$. Then $A$ and $U$ are unitary operators such that $A^2=U^2=\id$. For any $\alpha,\beta\in\mathbb R$ one has:
 $$\rho(Q(\alpha,\beta))=-\alpha A+2U-\beta\id.$$ If $|\alpha|+|\beta|<2$ then
$$-\alpha A+2U-\beta\id=U(2\id-\alpha U A-\beta U) $$ is invertible since $\|\alpha U A+\beta U\|\leqslant |\alpha|+|\beta|<2$. If $|\alpha|>|\beta|+2$ then
$$-\alpha A+2U-\beta\id=A(-\alpha\id +2AU-\beta A)$$ is invertible since $\|2AU-\beta A\|\leqslant 2+|\beta|<|\alpha|$. Finally, if $|\beta|>|\alpha|+2$ then $-\alpha A+2U-\beta\id$ is invertible since $\|-\alpha A+2U\|\leqslant |\alpha|+2<|\beta|$.
\end{proof}
\begin{proof}[Proof of Theorem \ref{ThGammaSpec}.] By construction, the spectrum of $\lambda_\mathcal{G}(4\Delta-e)$ coincides with the intersection of $\Sigma_{\lambda_\mathcal{G}}$ and the line $\alpha=-1$. By Lemma \ref{LmSigmaAny} we obtain $\sigma(\lambda_\mathcal{G}(4\Delta-e))\subset [-3,-1]\cup [1,3]$. It follows that $\sigma(\lambda_\mathcal{G}(\Delta))\subset [-\tfrac{1}{2},0]\cup [\tfrac{1}{2},1]$. The opposite inclusion follows from Theorem \ref{ThBGDelta}.
\end{proof}
In fact, calculations in \cite{BG} show that for any $t\in\mathbb R$ one has $$\sigma(\rho_x(-t a+b+c+d))=\Lambda_t:=(\{\alpha=t\}\cap \Omega)+1$$ (for instance, $\Lambda_t=[t-1,-t-1]\cup[t+3,-t+3]$ if $-2<t<0$) which is a union of two intervals, an interval, or two points (if $t=0$). Arguments similar to the proof of Theorem \ref{ThGammaSpec} show that $\sigma(\lambda_\mathcal{G}(-t a+b+c+d))=\Lambda_t$ for any $t\in\mathbb R$.

\subsection{Spectra of Koopman representations of $\mathcal{G}$.}

The boundary $\partial T$ of a binary rooted tree is homeomorphic to a space of infinite sequences $\{0,1\}^\mathbb{N}$ and hence is homeomorphic to a Cantor set. For any $q\in (0,1)$ define a measure $\nu_q$ on $\{0,1\}$ by $$\nu_q(\{0\})=q,\nu_q(\{1\})=1-q.$$ Let $\mu_q=\nu_q^\mathbb N$ be the corresponding Bernoulli measure on $\partial T$. 
For $q\in(0,1)$ let $\kappa_q$ be the Koopman representation associated to the action of $\mathcal{G}$ on $(\partial T,\mu_q)$ (this is the representation $\kappa_p$ with $p=(q,1-q)$ using the notations of Section \ref{SubsecAp}). We prove that the spectrum of $\kappa_q(\Delta)$ (see \eqref{EqDel}) does not depend on the parameter $q$ and thus coincides with the spectrum given by Theorem \ref{ThBGDelta}. Observe that the representations $\kappa_q$ for $q\neq\tfrac{1}{2}$ are irreducible (see \cite{DuGr15}), but $\kappa_{\frac{1}{2}}$ is a direct sum of countably many finite-dimensional irreducible representations (see \cite{BG}).
\begin{Th}\label{PropKoopGrig} For every $q\in (0,1)$ one has $\sigma(\kappa_q(\Delta))=[-\tfrac{1}{2},0]\cup[\tfrac{1}{2},1]$.
\end{Th}
 Fix $q\in (0,1),q\neq\tfrac{1}{2}$. Set $$A=\kappa_q(a),B=\kappa_q(b),C=\kappa_q(c),D=\kappa_q(d).$$ Consider the operators \begin{equation}\label{EqQab}Q(\alpha,\beta)=\kappa_q(4\Delta-(\alpha+1)a-(\beta+1))=-\alpha A+B+C+D-(\beta+1)\id\end{equation} on $L^2(\partial T,\mu_q)$. Denote by $\Sigma$ the set of pairs $(\alpha,\beta)\in\mathbb R^2$ such that $Q(\alpha,\beta)$ is not invertible. Theorem \ref{PropKoopGrig} is a consequence of the following:
 \begin{Prop}\label{PropSigma} $\Sigma=\Omega$.
 \end{Prop}
  In \cite{BG} the authors proved Proposition \ref{PropSigma} in the case $q=\tfrac{1}{2}$. For the proof they considered restrictions $Q_n(\alpha,\beta)$ of $Q(\alpha,\beta)$ on $\mathcal{G}$-invariant finite dimensional subspaces of $L^2(\partial T,\mu)$ and constructed operator recursions for $Q_n(\alpha,\beta)$ to describe spectra of $Q_n(\alpha,\beta)$ and $Q(\alpha,\beta)$. In the case $q\neq\frac{1}{2}$ the representation $\kappa_q$ is irreducible and so does not have invariant subspaces. We need to modify arguments from \cite{BG} and use operator recursions for infinite-dimensional Hilbert spaces.

Recall that $V_n$ is the set of vertices of level $n$ in $T$. For every $n$ encode vertices of $V_n$ by $\{0,1\}^n$ so that for every vertex $v=x_1x_2\ldots x_n\in V_n$ one has: $$\mu_q(\partial T_j)=q^{1-\sum x_i}(1-q)^{\sum x_i}.$$
 Let $v_0$ and $v_1$ be the vertices of the first level of $T$. Denote
$$\mathcal H=L^2(\partial T,\mu_q),\;\;\mathcal H_j=\{f\in\mathcal H:\supp(f)\subset \partial T_{v_j}\},\;\;\text{where}\;\;j=0,1.$$  Observe that $\mathcal H_0$ and $\mathcal H_1$ are isomorphic to $\mathcal H$ via the isometries $I_j:\mathcal H_j\to \mathcal H,j=0,1,$ given by:
$$(I_0f)(x)=\sqrt qf(0x),\;\;(I_1f)(x)=\sqrt{1-q}f(1x),$$ where $x\in\partial T$ is encoded by sequences from $\{0,1\}^\infty$.
Using the decomposition $\mathcal H=\mathcal H_0 \oplus \mathcal H_1$ and identifying with $\mathcal H$ the spaces $\mathcal H_i$ using the isometries $I_i$, $i=0,1$, we can write every operator on $\mathcal H$ in a $2\times 2$ block matrix form whose entries are also operators on $\mathcal H$.
The operators of the Koopman representation $\kappa_q$ corresponding to the generators of $\mathcal{G}$ can be written as follows:
\begin{gather}\label{EqRestr}\begin{split}A=\begin{bmatrix}0&\id\\ \id&0
\end{bmatrix},\;\;
B=\begin{bmatrix}A&0\\0&C
\end{bmatrix},\\
C=\begin{bmatrix}A&0\\0&D
\end{bmatrix},\;\;
D=\begin{bmatrix}\id&0\\0&B
\end{bmatrix}.\end{split}\end{gather} In particular, the recursions do not depend on parameter $q$. It follows that the operator $Q(\alpha,\beta)$ can be written as follows:
$$Q(\alpha,\beta)=\begin{bmatrix}2A-\beta\id&-\alpha \id \\-\alpha \id&B+C+D-(\beta+1)\id
\end{bmatrix}.$$ Notice that $(2A-\beta\id)(2A+\beta\id)=(4-\beta^2)\id$. Assume that $\beta\neq \pm 2$. Straightforward calculations show that
\begin{equation}\label{EqRec}Q(\alpha,\beta)\begin{bmatrix}\id&\frac{\alpha(2A+\beta\id)}{4-\beta^2}\\0&\id
\end{bmatrix}=
\begin{bmatrix}2A-\beta\id&0\\-\alpha\id&Q(\frac{2\alpha^2}{4-\beta^2},\beta+\frac{\alpha^2\beta}{4-\beta^2})
\end{bmatrix}.\end{equation}
Following \cite{BG} introduce a map on $\mathbb R^2\setminus\mathbb R\times\{\pm 2\}$  by
$$F(\alpha,\beta)=(\frac{2\alpha^2}{4-\beta^2},\beta+\frac{\alpha^2\beta}{4-\beta^2}).$$ Also, for $n\in\mathbb N$ set $(\alpha_n,\beta_n)=F^n(\alpha,\beta)$.
Since $\sigma(A)=\{-1,1\}$ we obtain
\begin{Lm}\label{PropInv} If $\beta\neq\pm 2$ then $(\alpha,\beta)\in\Sigma$ if and only if $F(\alpha,\beta)\in\Sigma$.
\end{Lm}
Next, let us prove another auxiliary statement.
\begin{Lm}\label{LmQ} $\Sigma\supset \{(\alpha,\beta):|\alpha-2|=|\beta|\}$.
\end{Lm}
\begin{proof}  Since $(B+C+D-\id)^2=4\id$ and clearly $B+C+D-\id$ is not a scalar operator we obtain that $\sigma(B+C+D-\id)=\{-2,2\}$. Thus, $-1$ and $3$ are eigenvalues of the operator $B+C+D$. Let $\eta\neq 0,\eta\in\mathcal H$ be such that $(B+C+D-3)\eta=0$. Let $\alpha=\beta+2$. Then $\alpha_1=\beta_1+2$ and $\beta_1=\frac{4\beta}{2-\beta}$. The map $$h(z)=\frac{4z}{2-z}$$ has two fixed points on the Riemann sphere: a repelling fixed point $0$ and an attracting fixed point $-2$, and for every $z\neq 0$  $h^n(z)\to -2$ exponentially fast. Thus, if $\beta\neq 0$, then $$\beta_n\to -2,\;\;\alpha_n\to 0$$ exponentially fast. It follows that $Q(\alpha_n,\beta_n)\eta\to 0$ exponentially fast.

Further, let $\xi\in\mathcal H$. Applying the operator in \eqref{EqRec} to the block vector $\begin{bmatrix}0\\ \xi\end{bmatrix}$, where $\xi\in\mathcal H$, we obtain:
$$ Q(\alpha,\beta)\begin{bmatrix}\frac{\alpha(2A+\beta\id)}{4-\beta^2}\xi\\ \xi\end{bmatrix}=
\begin{bmatrix}
0\\Q(\alpha_1,\beta_1)\xi
\end{bmatrix}.$$ Thus, there exists $\xi_1$ such that $\|\xi_1\|\geqslant\|\xi\|$ and $\|Q(\alpha_1,\beta_1)\xi\|=\|Q(\alpha,\beta)\xi_1\|$. By induction, we get that $\|Q(\alpha,\beta)\eta_n\|=\|Q(\alpha_n,\beta_n)\eta\|$ for some $\eta_n$ with $\|\eta_n\|\geqslant\|\eta\|$. Since $\|Q(\alpha_n,\beta_n)\eta\|$ converges to $0$, we obtain that $Q(\alpha,\beta)$ is not invertible. The case $\alpha=2-\beta$ can be treated similarly.
\end{proof}

\begin{proof}[Proof of Proposition \ref{PropKoopGrig}]
Following \cite{BG} consider the curves
$$\gamma_{n,j}=\{(\alpha,\beta):4-\beta^2+\alpha^2-4\alpha\cos(\tfrac{2\pi j}{2^n})=0\}.$$ Observe that $\gamma_{0,j}=\{(\alpha,\beta):|\alpha-2|=|\beta|\}$. Straightforward computations show that $F(\gamma_{n,j})\subset\gamma_{n-1,j}$ for all $n,j\in\mathbb N$. From Lemmas \ref{PropInv} and \ref{LmQ} taking into account that $\Sigma$ is closed we obtain that $\Sigma\supset \gamma_{n,j}$ for all $n,j$. Notice that the curve $\gamma_{n,j}$ can be written as:
$$\beta=\pm \sqrt{\alpha^2-4\alpha\cos(\tfrac{2\pi j}{2^n})+4}.$$ Since the union of curves $\gamma_{n,j}$ is dense in the region
$$S=\{(\alpha,\beta):||\alpha|-|\beta||\leqslant 2,\;||\alpha|+|\beta||\geqslant 2\}$$ we obtain that $\Sigma\supset S$. From Lemma \ref{LmSigmaAny} we deduce that $\Sigma=S$, which finishes the proof.\end{proof}

\subsection{Spectra of groupoid representations of $\mathcal{G}$.}
Let $\pi$ be the groupoid representation of $\mathcal{G}$ corresponding to the action of $\mathcal{G}$ on $(\partial T,\mu)$, where $\mu=\{\tfrac{1}{2},\tfrac{1}{2}\}^{\otimes\mathbb N}$ is the invariant Bernoulli measure on $\partial T$. The following Proposition follows from Theorem \ref{ThBGDelta} and Theorem \ref{ThMain}.
\begin{Prop}\label{PropGroupGrig} $\sigma(\pi(\Delta))=[-\tfrac{1}{2},0]\cup[\tfrac{1}{2},1]$.
\end{Prop} \noindent To give another illustration of the method of operator recursions we provide a sketch of a direct proof of Proposition \ref{PropGroupGrig}.
\begin{proof}
Let $v_0$ and $v_1$ be the vertices of the first level of $T$. For $i,j\in\{0,1\}$ introduce a subspace $$\mathcal H_{i,j}=\{\eta\in L^2(\mathcal R,\nu):\supp(\eta)\subset \partial T_{v_i}\times\partial T_{v_j}\}.$$ One has:
$$L^2(\mathcal R,\nu)=\mathcal H_{0,0}\oplus \mathcal H_{1,0}\oplus \mathcal H_{0,1}\oplus \mathcal H_{1,1}.$$ Recall that $x,y\in\partial T$ belong to the same orbit by $\mathcal{G}$ if and only if $x_i=y_i$ for all large enough $i$ (see \cite{Grig11}, Theorem 7.3). This implies that the subspaces $\mathcal H_{i,j}$ are canonically isomorphic to $L^2(\mathcal R,\nu)$. Thus, every operator acting on $L^2(\mathcal R,\nu)$ can be written in a $4\times 4$ block matrix form with entries operators on $L^2(\mathcal R,\nu)$. Set $$A=\pi(a),B=\pi(b),C=\pi(c),D=\pi(d).$$ It is straightforward to check that every operator $X$ from the latter list can be written as $$X=\begin{bmatrix}Y&0_2\\0_2&Y
\end{bmatrix},$$ where $0_2$ is the $2\times 2$ zero matrix and $Y$ is the $2\times 2$ block matrix  representation for the corresponding operator from \eqref{EqRestr}. Similarly to \eqref{EqQab} introduce an operator \begin{equation*}\label{EqQab1}Q(\alpha,\beta)=\pi(4\Delta-(\alpha+1)a-(\beta+1))=-\alpha A+B+C+D-(\beta+1)\id\end{equation*} in $L^2(\mathcal R,\nu)$ and denote by $\Sigma$ the set of pairs $(\alpha,\beta)\in\mathbb R^2$ such that $Q(\alpha,\beta)$ is not invertible. One has:
 $$Q(\alpha,\beta)=\begin{bmatrix}Y&0_2\\0_2&Y
\end{bmatrix},\;\;\text{where}\;\;Y=\begin{bmatrix}2A-\beta\id&-\alpha \id \\-\alpha \id&B+C+D-(\beta+1)\id
\end{bmatrix}.$$ Similarly to Proposition \ref{PropSigma} one can show that $\Sigma=\Omega$ from which the statement of Proposition \ref{PropGroupGrig} follows easily.
\end{proof}
\subsection*{Acknowledgement} The authors are grateful to Maria Gabriella Kuhn for useful discussions.


\end{document}